\documentclass{llncs}
\usepackage{amsmath,amssymb}
\usepackage{cases}

\begin{document}

\newcommand{\HALversion}[1]{#1}
\newcommand{\ALTversion}[1]{}

\mainmatter

\title{Lipschitz Bandits without the Lipschitz Constant}
\author{S{\'e}bastien Bubeck\inst{1} \and Gilles Stoltz\inst{2,3}
\and Jia Yuan Yu\inst{2,3}}
\institute{Centre de Recerca Matem{\`a}tica, Barcelona, Spain
\and
Ecole normale sup{\'e}rieure, CNRS, Paris, France
\and
HEC Paris, CNRS, Jouy-en-Josas, France}
\maketitle

\newcommand{\R}{\mathbb{R}}
\newcommand{\cX}{\mathcal{X}}
\newcommand{\cF}{\mathcal{F}}
\newcommand{\cB}{\mathcal{B}}
\newcommand{\of}{\overline{f}}
\newcommand{\oL}{\overline{L}}
\newcommand{\oR}{\overline{R}}
\newcommand{\E}{\mathbb{E}}
\renewcommand{\d}{\mbox{d}}
\renewcommand{\geq}{\geqslant}
\renewcommand{\leq}{\leqslant}
\newcommand{\wh}{\widehat}
\newcommand{\wt}{\widetilde}
\newcommand{\ux}{\underline{x}}
\newcommand{\uy}{\underline{y}}
\newcommand{\uk}{\underline{k}}
\newcommand{\uK}{\underline{K}}
\newcommand{\uI}{\underline{I}}
\newcommand{\uu}{\underline{u}}
\newcommand{\us}{\underline{s}}
\newcommand{\uone}{\underline{1}}
\newcommand{\transpose}{\mbox{\scriptsize{T}}}
\newcommand{\norm}[1][\,\cdot\,]{\ensuremath{\Arrowvert #1 \Arrowvert}}
\newcommand{\eps}{\varepsilon}

\newtheorem{assumption}{Assumption}

\begin{abstract}
We consider the setting of stochastic bandit problems with a continuum of arms
indexed by $[0,1]^d$.
We first point out that the strategies considered so far in the literature
only provided theoretical guarantees of the form: given some tuning parameters,
the regret is small with respect to a
class of environments that depends on these
parameters. This is however not the right perspective, as it is the strategy
that should adapt to the specific bandit environment at hand, and not the other way
round. Put differently, an adaptation issue is raised. We solve it for the special case
of environments whose mean-payoff functions are globally Lipschitz. More
precisely, we show that the minimax optimal orders of magnitude
$L^{d/(d+2)} \, T^{(d+1)/(d+2)}$ of the regret bound over $T$ time instances against
an environment whose mean-payoff function $f$ is Lipschitz with constant $L$
can be achieved without knowing $L$ or $T$ in advance. This is in contrast
to all previously known strategies, which require to some extent the knowledge
of~$L$ to achieve this performance guarantee.
\end{abstract}

\section{Introduction}

In the (stochastic) bandit problem, a gambler tries to maximize the revenue
gained by sequentially playing one of a finite number of arms
that are each associated with initially unknown
(and potentially different) payoff distributions~\cite{Rob52}.
The gambler selects and pulls arms one by one in a sequential manner,
simultaneously learning about the machines' payoff-distributions and accumulating rewards (or losses).
Thus, in order to maximize his gain, the gambler must choose the next arm by
taking into consideration both the urgency of gaining reward (``exploitation'')
and acquiring new information (``exploration'').
Maximizing the total cumulative payoff is equivalent to minimizing
the (total) {\em regret}, that is, minimizing the difference between the total cumulative payoff of the
gambler and that of another clairvoyant gambler who chooses the arm with the best mean-payoff in every
round. The quality of the gambler's strategy can be characterized by the rate of growth of his expected
regret with time. In particular, if this rate of growth is sublinear, the gambler in the long run plays
as well as his clairvoyant counterpart. \smallskip

\noindent
\textbf{Continuum-armed bandit problems.}
Although the early papers studied bandits with a finite number of arms, researchers soon realized
that bandits with infinitely many arms are also interesting,
as well as practically significant. One particularly important case is when the arms are identified
by a finite number of continuous-valued parameters,
resulting in online optimization problems over continuous finite-dimensional spaces.
During the last decades numerous contributions have investigated  such continuum-armed bandit
problems, starting from the early formulations of \cite{Agr95b,Cop04,Kle04}
to the more recent approaches of \cite{AOS07,KSU08,HOO}.
A special case of interest, which forms a bridge between
the case of a finite number of arms and the continuum-armed setting, is
the problem of bandit linear optimization, see~\cite{DHK08b} and the references therein.
\smallskip

\noindent
\textbf{Not the right perspective!}
We call an environment $f$ the mapping that associates with each arm $x \in \cX$
the expectation $f(x)$ of its associated probability distribution.
The theoretical guarantees given in the literature mentioned above
are of the form: given some tuning parameters, the strategy is competitive, and sometimes
even minimax optimal, with respect to a large class of environments
that unfortunately depends on these parameters. But of course, this is not the
right perspective: it is the strategy that should adapt to the environment, not the other
way round!

More precisely, these parameters describe the smoothness
of the environments $f$ in the class at hand
in terms of a global regularity and/or local regularities around the global maxima of $f$.
The issues raised by some of the works mentioned above can be roughly described as follows:
\begin{itemize}
\item The class of environments for the CAB1 algorithm of~\cite{Kle04} is formed by environments that are
$(\alpha,L,\delta)$--uniformly locally Lipschitz and the strategy CAB1 needs to know $\alpha$
to get the optimal dependency in the number $T$ of arms pulled;
\item For the Zooming algorithm of~\cite{KSU08}, it is formed by environments that are 1--Lipschitz with respect to a fixed
and known metric $L$;
\item The HOO algorithm of~\cite{HOO} basically needs to know the pseudo-metric $\ell$
with respect to which $f$ is weakly Lipschitz continuous, with Lipschitz constant
equal to 1;
\item Other examples include the UCB-air algorithm (which relies on a smoothness parameter $\beta$, see~\cite{WAM09}),
the OLOP algorithm (smoothness parameter $\gamma$, see~\cite{BM10}), the LSE algorithm (smoothness parameter $C_L$, see~\cite{YM11}),
the algorithm presented in~\cite{AOS07} and so on.
\end{itemize}
\smallskip

\noindent
\textbf{Adaptation to the unknown smoothness is needed.}
In a nutshell, adaptive methods are required. By adaptive methods, we mean---as is done in the statistical literature---agnostic methods, i.e., with minimal prior knowledge about $f$,
that nonetheless obtain almost the same
performance against a given environment $f$ as if its smoothness were known beforehand.

More precisely,
given a fixed (possibly vector-valued) parameter $L$ lying in a set $\mathcal{L}$
and a class of allowed environments $\cF_L$, where $L \in \mathcal{L}$,
existing works present algorithms that are such that their worst-case regret bound
over $T$ time steps against environments in $\cF_L$,
\[
\sup_{f \in \cF_L} \, R_T(f) \leq \varphi(T,L) \,,
\]
is small and even minimax optimal, i.e., such that it has the optimal dependencies
on $T$ and $L$. However, to do so, the knowledge of $L$ is required.
In this work, we are given a much larger class of environments
$\cF = \cup_{L \in \mathcal{L}} \,\, \cF_L$, and our goal is an algorithm that adapts in
finite time to every instance $f$ of $\cF$, in the sense that for all $T$ and $f \in \cF$,
the regret $R_T(f)$ is at most of the order of $\min \varphi(T,L)$, where the minimum is
over the parameters $L$ such that $f \in \cF_L$.

Since we are interested in worst-case bounds, we will have to consider {dis\-tri\-bu\-tion}-free bounds
(i.e., bounds that only depend on a given class $\cF_L$);
of course, the orders of magnitude of the latter, even when they are minimax optimal,
are often far away --as far as the dependencies in $T$ are concerned-- with respect to
distribution-dependent bounds (i.e., bounds that may depend on a specific instance $f \in \cF_L$).
\smallskip

\noindent
\textbf{Links with optimization algorithms.}
Our problem shares some common points with the maximization of a deterministic function $f$
(but note that in our case, we only get to see noisy observations of the values of $f$).
When the Lipschitz constant $L$ of $f$ is known, an approximate maximizer can be found
with well-known Lipschitz optimization algorithms (e.g., Shubert's algorithm).
The case of unknown $L$ has been studied in \cite{Jones93,Horn05}.
The DIRECT algorithm of \cite{Jones93}
carries out Lipschitz optimization by using the
smallest Lipschitz constant that is consistent with the observed data;
although it works well in practice, only asymptotic convergence can be guaranteed. The algorithm of \cite{Horn05} iterates over an increasing sequence of possible values of $L$; under an additional assumption on the minimum increase in the neighborhood of the maximizers, it guarantees a worst-case error of the order of $L^2T^{-2/d}$ after taking $T$ samples of the deterministic function~$f$.
\medskip

\noindent
\textbf{Adaptation to a global Lipschitz smoothness in bandit problems.}
We provide in this paper a first step toward a general theory of adaptation.
To do so, we focus on the special case of classes $\cF_L$ formed by
all environments $f$ that are $L$--Lipschitz with respect to the supremum norm
over a subset of $\R^d$: the hypercube $[0,1]^d$ for simplicity.
This case covers partially the settings of~\cite{HOO} and~\cite{KSU08},
in which the Lipschitz constant was equal to 1, a fact known by the algorithms.
(Extensions to H{\"o}lderian-type assumptions as in~\cite{Kle04,AOS07} will be considered in
future work.)

As it is known, getting the minimax-optimal dependency on $T$ is easy, the difficult part
is getting that on $L$ without knowing the latter beforehand.
\smallskip

\noindent
\textbf{Our contributions.}
Our algorithm proceeds by discretization as in~\cite{Kle04}. To determine the correct
discretization step, it first resorts to an uniform exploration yielding
a rather crude estimate of the Lipschitz constant (that is however sufficient for our needs);
in a second phase, it finds the optimal interval
using a standard exploration-exploitation strategy.
Our main assumptions are (essentially) that $f$ and its derivative are Lipschitz continuous in the hypercube.

We feel that this two-step approach can potentially be employed in more general settings well beyond ours:
with the notation above,
the uniform-exploration phase performs a model-selection step and recommends a class $\cF_{\wt{L}}$,
which is used in the second phase to run a continuum-armed bandit strategy tuned with the optimal parameters
corresponding to $\wt{L} \in \mathcal{L}$. However, for the sake of simplicity, we study only a particular
case of this general metho\-dology.
\medskip

\noindent
\textbf{Outline of the paper.}
In Section~\ref{sec:set}, we describe the setting and the classes of environments of interest,
establish a minimax lower bound on the achievable performance (Section~\ref{sec:minimax}),
and indicate how to achieve it when the global Lipschitz parameter $L$ is known (Section~\ref{sec:Lknown}).
Our main contribution (Section~\ref{sec:main}) is then a method to achieve it when the Lipschitz constant is unknown
for a slightly restricted class of Lipschitz functions.

\section{Setting and notation}
\label{sec:set}

We consider a $d$--dimensional compact set of arms, say, for simplicity, $\cX = [0,1]^d$,
where $d \geq 1$. With each arm $\ux \in [0,1]^d$
is associated a probability distribution $\nu_{\ux}$ with known bounded support, say $[0,1]$;
this defines an environment.
A key quantity of such an environment is
given by the expectations $f(\ux)$ of the distributions $\nu_{\ux}$.
They define a mapping $f : [0,1]^d \to [0,1]$, which we call the mean-payoff function.

At each round $t \geq 1$, the player chooses an arm $\uI_t \in [0,1]^d$
and gets a reward $Y_t$ sampled independently from $\nu_{\uI_t}$ (conditionally
on the choice of $\uI_t$). We call a strategy the (possibly randomized) rule that
indicates at each round which arm to pull given the history of past rewards.
\medskip

We write the elements $\ux$ of $[0,1]^d$ in columns;
$\ux^{\transpose}$ will thus denote a row vector with $d$ elements.

\begin{assumption}
\label{ass:1}
{\rm
We assume that $f$ is twice differentiable, with Hessians uniformly
bounded by $M$ in the following sense: for all $\ux \in [0,1]^d$
and all $\uy \in [0,1]^d$,
\[
\Bigl| \uy^{\transpose} \, H_f(\ux) \,\, \uy \Bigr| \leq M \, \norm[\uy]_\infty^2\,.
\]
The $\ell^1$--norm of the gradient $\norm[\nabla f]_1$ of $f$
is thus continuous and it achieves its maximum on $[0,1]^d$, whose value
is denoted by $L$. As a result, $f$ is Lipschitz with respect to the $\ell^\infty$--norm with constant $L$
(and $L$ is the smallest\footnote{The proof of the
approximation lemma will show why this is the case.}
constant for which it is Lipschitz): for all $\ux,\uy \in [0,1]^d$,
\[
\bigl| f(\ux) - f(\uy) \bigr| \leq L \, \norm[\ux - \uy]_\infty\,.
\]
}
\end{assumption}

In the sequel we denote by $\cF_{L,M}$ the set of environments whose mean-payoff functions
satisfy the above assumption.

We also denote by $\cF_{L}$ the larger set of environments whose
mean-payoff functions $f$ is only constrained to be $L$--Lipschitz with respect to the $\ell^\infty$--norm.

\subsection{The minimax optimal orders of magnitude of the regret}
\label{sec:minimax}

When $f$ is continuous, we denote by
\[
f^\star = \sup_{\ux \in [0,1]^d} f(\ux) = \max_{\ux \in [0,1]^d} f(\ux)
\]
the largest expected payoff in a single round. The expected regret $\oR_T$
at round $T$ is then defined as
\[
\oR_T = \E \! \left[ T f^\star - \sum_{t=1}^T Y_t \right]
= \E \! \left[ T f^\star - \sum_{t=1}^T f(\uI_t) \right]
\]
where we used the tower rule and where the expectations
are with respect to the random draws of the $Y_t$ according to
the $\nu_{\uI_t}$ as well as to any auxiliary randomization
the strategy uses.

In this article, we are interested in controlling the worst-case
expected regret over all environments of $\cF_L$.
The following minimax lower bound follows from a
straightforward adaptation of the proof of~\cite[Theorem~13]{HOO},
which\HALversion{ is provided in Section~\ref{sec:th:LB} in appendix.}\ALTversion{ is
omitted from this extended abstract and may be found in~\cite{BSYHAL}.}
(The adaptation is needed because the hypothesis on
the packing number is not exactly satisfied in the form stated in~\cite{HOO}.)

\begin{theorem}
\label{th:LB}
For all strategies of the player and for all
\[
T \geq \max \left\{ L^d, \,\, \left( \frac{0.15 \, L^{2/(d+2)}}{\max\{d,2\}} \right)^{\!\! d} \, \right\},
\]
the worst-case regret over the set $\cF_{L}$ of all environments that are
$L$--Lipschitz with respect to the $\ell^\infty$--norm is larger than
\[
\sup_{\cF_L} \oR_T \geq 0.15 \, L^{d/(d+2)}\,T^{(d+1)/(d+2)}\,.
\]
\end{theorem}

The multiplicative constants are not optimized in this bound (a more careful proof
might lead to a larger constant in the lower bound).

\subsection{How to achieve a minimax optimal regret when $L$ is known}
\label{sec:Lknown}

In view of the previous section,
our aim is to design strategies with worst-case expected regret
$\sup_{\cF_{L}} \oR_T$ less than something of order $L^{d/(d+2)}\,T^{(d+1)/(d+2)}$
when $L$ is unknown.
A simple way to do so when $L$ is known was essentially proposed in the introduction of~\cite{Kle04} (in the case $d=1$);
it proceeds by discretizing the arm space.
The argument is reproduced below and can be used even when $L$ is unknown to recover the
optimal dependency $T^{(d+1)/(d+2)}$ on $T$ (but then, with a suboptimal dependency on $L$).
\medskip

We consider the approximations $\of_m$ of $f$ with $m^d$ regular hypercube bins in the
$\ell^\infty$--norm, i.e., $m$ bins are formed in each direction and combined
to form the hypercubes. Each of these hypercube bins is indexed by an
element $\uk = (k_1,\ldots,k_d) \in \{ 0,\ldots, m-1 \}^d$.
The average value of $f$ over the bin indexed by $\uk$ is denoted by
\[
\of_m(\uk) = m^d \int_{\uk/m+[0,1/m]^d} f(\ux) \,\d \ux\,.
\]

We then consider the following two-stage strategy, which is based
on some strategy MAB for multi-armed bandits; MAB will refer to a generic strategy
but we will instantiate below the obtained bound.
Knowing $L$ and assuming that $T$ is fixed and known in advance,
we may choose beforehand $m = \bigl\lceil L^{2/(d+2)} T^{1/(d+2)} \bigr\rceil$.
The decomposition of $[0,1]^d$ into $m^d$ bins thus obtained will play
the role of the finitely many arms of the multi-armed bandit problem.
At round $t \geq 1$, whenever the MAB strategy prescribes to pull
bin $\uK_t \in \{ 0,\ldots,m-1 \}^d$, then first, an arm $\uI_t$ is pulled at random in the hypercube
$\uK_t/m + [0,1/m]^d$;
and second, given $\uI_t$, the reward $Y_t$ is drawn at random
according to $\nu_{\uI_t}$. Therefore, given $\uK_t$, the reward $Y_t$ has an expected value of $\of_m(\uK_t)$.
Finally, the reward $Y_t$ is returned to the underlying MAB strategy.

Strategy MAB is designed to control the regret with respect to the best of the $m^d$ bins,
which entails that
\[
\E \left[ T \max_{\uk} \of_m(\uk) - \sum_{t=1}^T Y_t \right] \leq \psi\bigl(T,m^d\bigr)\,,
\]
for some function $\psi$ that depends on MAB.
Now, whenever $f$ is $L$--Lipschitz with respect to the $\ell^\infty$--norm,
we have that for all $\uk \in \{ 0,\ldots,m-1 \}^d$ and
all $\ux \in \uk/m + [0,1/m]^d$, the difference $\bigl| f(\ux) - \of_m(\uk) \bigr|$
is less than $L/m$; so that\footnote{Here, one could object that we only use the local Lipschitzness of $f$
around the point where it achieves its maximum; however, we need $f$ to be $L$--Lipschitz in
an $1/m$--neighborhood of this maximum, but the optimal value of $m$ depends on $L$. To
solve the chicken-egg problem, we restricted our attention to globally $L$--Lipschitz
functions, which, anyway, in view of Theorem~\ref{th:LB}, comes at no cost
as far as minimax-optimal orders of magnitude of the regret bounds in $L$ and $T$ are
considered.}
\[
\max_{\ux \in [0,1]^d} f(x) - \max_{\uk} \of_m(\uk) \leq \frac{L}{m}\,.
\]
All in all, for this MAB-based strategy, the regret is bounded by the sum of
the approximation term $L/m$ and of the regret term for multi-armed bandits,
\begin{multline}
\label{eq:bd}
\sup_{\cF_L} \oR_T \leq T \frac{L}{m} + \psi(T,m^d) \\ \leq L^{d/(d+2)} T^{(d+1)/(d+2)}
+ \psi\Bigl(T, \bigl( \bigl\lceil L^{2/(d+2)} T^{1/(d+2)} \bigr\rceil \bigr)^d \Bigr)\,.
\end{multline}
We now instantiate this bound.

The INF strategy of~\cite{AB10} (see also~\cite{ABL11}) achieves $\psi(T,m') = 2 \sqrt{2 T m'}$
and this entails a final $O \bigl( L^{d/(d+2)}\,T^{(d+1)/(d+2)} \bigr)$ bound in~(\ref{eq:bd}).
Note that for the EXP3 strategy of~\cite{ACFS02} or
the UCB strategy of~\cite{ACF02}, extra logarithmic terms of the order of $\ln T$
would appear in the bound.

\section{Achieving a minimax optimal regret not knowing $L$}
\label{sec:main}

In this section, our aim is to obtain a worst-case regret bound of the minimax-optimal order of
$L^{d/(d+2)}\,T^{(d+1)/(d+2)}$ even when $L$ is unknown.
To do so, it will be useful to first estimate $L$; we will provide a (rather crude) estimate
suited to our needs, as our goal is the minimization of the regret rather than the best
possible estimation of $L$. Our method is based on the following approximation results.

For the estimation to be efficient, it will be convenient to restrict our attention
to the subset $\cF_{L,M}$ of $\cF_L$, i.e., we will consider the additional assumptions
on the existence and boundedness of the Hessians asserted in Assumption~\ref{ass:1}.
However, the obtained regret bound (Theorem~\ref{th:main}) will suffer from some
(light) dependency on $M$ but will have the right orders of magnitude in $T$ and $L$;
the forecaster used to achieve it depends neither on $L$ nor on $M$ and is fully adaptive.

\subsection{Some preliminary approximation results}

We still consider the approximations $\of_m$ of $f$ over $[0,1]^d$ with $m^d$ regular bins.
We then introduce the following approximation of $L$:
\[
\oL_m = m \max_{\uk \in \{1,\hdots,m-2\}^d} \, \max_{\us \in \{-1,1\}^d} \,
\Bigl| \of_m(\uk) - \of_m(\uk + \us) \Bigr|\,.
\]
This quantity provides a fairly good approximation of the Lipschitz constant, since $m\left(\of_m(\uk) - \of_m(\uk + \us)\right)$  is an estimation of the (average) derivative of $f$ in bin $\uk$ and direction $\us$.

The lemma below relates precisely $\oL_m$ to $L$: as $m$ increases, $\oL_m$ converges to~$L$.

\begin{lemma}
\label{lm:approx}
If $f \in \cF_{L,M}$ and $m \geq 3$, then
\[
L - \frac{7 M}{m} \leq \oL_m \leq L\,.
\]
\end{lemma}

\begin{proof}
We note that for all $\uk \in \{1,\hdots,m-2\}^d$ and $\us \in \{-1,1\}^d$,
we have by definition
\begin{multline}
\nonumber
\Bigl| \of_m(\uk) - \of_m(\uk + \us) \Bigr|
= m^d \left| \int_{\uk/m + [0,1/m]^d} \Bigl( f(\ux) - f\bigl(\ux + \us/m\bigr) \Bigr) \,\d \ux \, \right| \\
\leq m^d \int_{\uk/m + [0,1/m]^d} \Bigl| f(\ux) - f\bigl(\ux + \us/m\bigr) \Bigr| \,\d \ux\,.
\end{multline}
Now, since $f$ is $L$--Lipschitz in the $\ell^\infty$--norm, it holds that
\[
\Bigl| f(\ux) - f\bigl(\ux + \us/m\bigr) \Bigr|
\leq L \bigl\Arrowvert \us/m \bigr\Arrowvert_\infty = \frac{L}{m}\,;
\]
integrating this bound entails the stated upper bound $L$ on $\oL_m$. \medskip

For the lower bound, we first denote by $\ux_\star \in [0,1]^d$ a
point such that $\bigl\Arrowvert \nabla f(\ux_\star) \bigr\Arrowvert_1 = L$.
(Such a point always exists, see Assumption~\ref{ass:1}.)
This point belongs to some bin in $\{ 0,\ldots,m-1\}^d$; however, the
closest bin $\uk^\star_m$ in $\{1,\hdots,m-2\}^d$ is such that
\begin{equation}
\label{eq:kstar}
\forall \ux \in \uk^\star_m/m +[0,1/m]^d, \qquad
\Arrowvert \ux - \ux_\star \Arrowvert_{\infty} \leq \frac{2}{m}.
\end{equation}
Note that this bin $\uk^\star_m$ is such that all $\uk^\star_m + \us$
belong to $\{ 0,\ldots,m-1 \}^d$ and hence legally index hypercube bins,
when $\us \in \{-1,1\}^d$.
Now, let $\us^\star_m \in \{-1,1\}^d$ be such that
\begin{equation}
\label{eq:nabla}
\nabla f(\ux_\star) \,\cdot\, \us^\star_m =  \bigl\Arrowvert \nabla f(\ux_\star) \bigr\Arrowvert_1 = L\,,
\end{equation}
where $\,\cdot\,$ denotes the inner product in $\mathbb{R}^d$.
By the definition of $\oL_m$ as some maximum,
\begin{multline}
\label{eq:1}
\oL_m \geq m \, \Bigl| \of_m\bigl(\uk^\star_m\bigr) - \of_m\bigl(\uk^\star_m + \us^\star_m\bigr) \Bigr| \\
= m \times m^{d} \, \left| \int_{\uk^\star_m/m + [0,1/m]^d} \Bigl( f(\ux) -
f\bigl(\ux + \us^\star_m/m \bigr) \Bigr) \,\d \ux \, \right|.
\end{multline}

Now, Taylor's theorem (in the mean-value form for real-valued twice differentiable
functions of possibly several variables) shows that for
any $\ux \in \uk^\star_m/m + [0,1/m]^d$, there exists two elements $\xi$ and $\zeta$,
belonging respectively to the segments between $\ux$ and $\ux_\star$, on the one hand,
between $\ux_\star$ and $\ux+\us_m^\star/m$ on the other hand, such that
\begin{align*}
& f(\ux) - f\bigl(\ux + \us^\star_m/m \bigr) \\
& = \bigl( f(\ux) - f(\ux_\star) \bigr) + \Bigl( f(\ux_\star) - f\bigl(\ux + \us^\star_m/m \bigr) \Bigr) \\
& = \nabla f(\ux_\star) \,\cdot\, (\ux - \ux_\star)
+ \frac{1}{2} (\ux - \ux_\star)^{\transpose} \,\, H_f(\xi) \,\, (\ux - \ux_\star) \\
& \quad - \nabla f(\ux_\star) \,\cdot\, \bigl(\ux + \us^\star_m/m - \ux_\star \bigr)
- \frac{1}{2} \bigl(\ux + \us^\star_m/m - \ux_\star \bigr)^{\transpose} \,\, H_f(\zeta) \,\,
\bigl(\ux + \us^\star_m/m - \ux_\star \bigr) \\
& = - \nabla f(\ux_\star) \,\cdot\, \frac{\us^\star_m}{m} \ +
\frac{1}{2} (\ux - \ux_\star)^{\transpose} \,\, H_f(\xi) \,\, (\ux - \ux_\star) \\
& \hspace{3cm} - \frac{1}{2} \bigl(\ux + \us^\star_m/m - \ux_\star \bigr)^{\transpose} \,\, H_f(\zeta) \,\,
\bigl(\ux + \us^\star_m/m - \ux_\star \bigr)\,.
\end{align*}
Using~\eqref{eq:nabla} and substituting the bound on the Hessians stated in Assumption~\ref{ass:1}, we get
\[
f(\ux) - f\bigl(\ux + \us^\star_m/m \bigr) \leq - \frac{L}{m} + \frac{M}{2} \norm[\ux - \ux_\star]_\infty^2
+ \frac{M}{2} \bigl\Arrowvert \ux + \us^\star_m/m - \ux_\star \bigr\Arrowvert_\infty^2\,;
\]
substituting~\eqref{eq:kstar}, we get
\[
f(\ux) - f\bigl(\ux + \us^\star_m/m \bigr) \leq - \frac{L}{m} + \frac{M}{2 m^2} \bigl( 2^2 + 3^2 \bigr) \leq - \frac{L}{m} + \frac{7 M}{m^2} \leq 0\,,
\]
where the last inequality holds with no loss of generality (if it does not, then the lower bound on $\oL_m$
in the statement of the lemma is trivial).
Substituting and integrating this equality in~\eqref{eq:1} and using the triangle inequality, we get
\begin{eqnarray*}
\oL_m \geq L - \frac{7 M}{m}.
\end{eqnarray*}
This concludes the proof.
\qed
\end{proof}

\subsection{A strategy in two phases}

Our strategy is described in Figure~\ref{fig:strat}; several
notation that will be used in the statements and proofs of
some results below are defined therein. Note that we proceed in two phases:
a pure exploration phase, when we estimate $L$ by some $\wt{L}_m$, and an exploration--exploitation
phase, when we use a strategy designed for the case of finitely-armed bandits on
a discretized version of the arm space. (The discretization step depends on the estimate
obtained in the pure exploration phase.)
\begin{figure}[p]
\rule{\linewidth}{.5pt}
{\small
\emph{Parameters}:
\begin{itemize}
\item Number $T$ of rounds;
\item Number $m$ of bins (in each direction) considered in the pure exploration phase;
\item Number $E$ of times each of them must be pulled;
\item A multi-armed bandit strategy MAB (taking as inputs a number $m^d$
of arms and possibly other parameters).
\end{itemize}
\medskip

\emph{Pure exploration phase}:
\begin{enumerate}
\item For each $\uk \in \{ 0,\ldots,m-1 \}^d$
\begin{itemize}
\item pull $E$ arms independently uniformly
at random in $\uk/m + [0, 1/m]^d$ and get $E$ associated rewards
$Z_{\uk,j}$, where $j \in \{ 1,\ldots,E \}$;
\item compute the average reward for bin $\uk$,
\[
\wh{\mu}_{\uk} = \frac{1}{E} \sum_{j = 1}^E Z_{\uk,j}\,;
\]
\end{itemize}
\item Set
\[
\wh{L}_m = m \max_{\uk \in \{1,\hdots,m-2\}^d} \max_{\us \in \{-1,1\}^d}
\bigl| \wh{\mu}_{\uk} - \wh{\mu}_{\uk + \us} \bigr|
\vspace{.15cm}
\]
and define $\ \displaystyle{\wt{L}_m = \wh{L}_m + m \sqrt{\frac{2}{E} \ln (2 m^d T)}} \ $
as well as $\ \wt{m} = \Bigl\lceil \wt{L}_m^{2/(d+2)} T^{1/(d+2)} \Bigr\rceil$.
\end{enumerate}
\medskip

\emph{Exploration--exploitation phase}: \smallskip \\
Run the strategy MAB with $\wt{m}^d$ arms
as follows; for all $t = Em+1,\ldots,T$,
\begin{enumerate}
\item If MAB prescribes to play arm $\uK_t \in \bigl\{ 0,\ldots,\wt{m}-1 \bigr\}^d$,
pull an arm $\uI_t$ at random in $\uK_t/m + [0, 1/m]^d$;
\item Observe the associated payoff $Y_t$, drawn independently according to $\nu_{\uI_t}$;
\item Return $Y_t$ to the strategy MAB.
\end{enumerate}
}
\rule{\linewidth}{.5pt}
\caption{\label{fig:strat} The considered strategy.}
\end{figure}

The first step in the analysis is to relate $\wt{L}_m$ and $\wh{L}_m$
to the quantity they are estimating, namely $\oL_m$.

\begin{lemma}
\label{lm:conc}
With probability at least $1-\delta$,
\[
\bigl| \wh{L}_m - \oL_m \bigr| \leq m \sqrt{\frac{2}{E} \ln \frac{2 m^d}{\delta}}\,.
\]
\end{lemma}

\begin{proof}
We consider first a fixed $\uk \in \{ 0,\ldots,m-1 \}^d$;
as already used in Section~\ref{sec:Lknown}, the $Z_{\uk,j}$
are independent and identically distributed according to a distribution
on $[0,1]$ with expectation $\of_m(\uk)$, as $j$ varies between $1$ and $E$.
Therefore, by Hoeffding's inequality, with probability at least $1-\delta/m^d$
\[
\bigl| \wh{\mu}_{\uk} - \of_m(\uk) \bigr| \leq \sqrt{\frac{1}{2E} \ln \frac{2m^d}{\delta}}\,.
\]

Performing a union bound and using the triangle inequality, we get that with
probability at least $1-\delta$,
\[
\forall \uk,\uk' \in \{ 0,\ldots,m-1 \}, \qquad \Bigl| \,\, \bigl| \wh{\mu}_{\uk} - \wh{\mu}_{\uk'} \bigr| -
\bigl| \of_m(\uk) - \of_m(\uk') \bigr| \,\, \Bigr| \leq \sqrt{\frac{2}{E} \ln \frac{2 m^d}{\delta}}\,.
\]
This entails the claimed bound.
\qed
\end{proof}

By combining Lemmas~\ref{lm:approx} and~\ref{lm:conc}, we
get the following inequalities on $\wt{L}_m$, since the latter
is obtained from $\wh{L}_m$ by adding a deviation term.

\begin{corollary}
\label{cor:wtL}
If $f \in \cF_{L,M}$ and $m \geq 3$, then, with probability at least $1-1/T$,
\[
L - \frac{7 M}{m} \leq \wt{L}_m \leq L + 2 m\sqrt{\frac{2}{E} \ln \bigl( 2 m^d T \bigr)}\,.
\]
\end{corollary}

We state a last intermediate result; it relates the regret of the strategy
of Figure~\ref{fig:strat} to the regret of the strategy MAB that it takes
as a parameter.

\begin{lemma}
\label{lm:3}
Let $\psi(T',m')$ be a distribution-free upper bound on the expected
regret of the strategy MAB, when run for $T'$ rounds on a multi-armed bandit
problem with $m'$ arms, to which payoff distributions over $[0,1]$ are associated.
The expected regret of the strategy defined in Figure~\ref{fig:strat} is then
bounded from above as
\[
\sup_{\cF_L} \oR_T \leq E m^d +
\E \left[ \frac{L T}{\wt{m}} + \psi \bigl( T-Em^d, \, \wt{m}^d \bigr) \right].
\]
\end{lemma}

\begin{proof}
As all payoffs lie in $[0,1]$, the regret during the pure exploration phase is bounded by the
total length $E m^d$ of this phase.

Now, we bound the (conditionally) expected regret of the MAB strategy
during the exploration--exploitation phase; the conditional expectation
is with respect to the pure exploration phase and is used to fix the value of $\wt{m}$.
Using the same arguments as in Section~\ref{sec:Lknown},
the regret during this phase, which lasts $T-E m^d$ rounds, is bounded against any environment
in $\cF_{L}$ by
\[
L \frac{T - E m^d}{m} + \psi \bigl( T-Em^d, \, \wt{m}^d \bigr)\,.
\]
The tower rule concludes the proof.
\qed
\end{proof}

We are now ready to state our main result. Note that it is somewhat
unsatisfactory as the main regret bound~\eqref{eq:bd2}
could only be obtained on the restricted class~$\cF_{L,M}$
and depends (in a light manner, see comments below) on the parameter $M$,
while having the optimal orders of magnitude in $T$ and $L$.
These drawbacks might be artifacts of the analysis
and could be circumvented, perhaps in the light of the proof of the lower bound (Theorem~\ref{th:LB}),
which exhibits the worst-case elements of $\cF_L$ (they seem to belong to
some set $\cF_{L,M_L}$, where $M_L$ is a decreasing function of $L$).

Note howewer that the dependency on $M$ in~\eqref{eq:bd2} is in the additive
form. On the other hand, by adapting the argument of Section~\ref{sec:Lknown},
one can prove distribution-free regret bounds on $\cF_{L,M}$ with an improved order of magnitude as far
as $T$ is concerned but at the cost of getting $M$ as a multiplicative constant in the picture.
While the corresponding bound might be better in some regimes, we consider here
$\cF_{L,M}$ instead of $\cF_L$ essentially to remove pathological functions,
and as such we want the weakest dependency on $M$.

\begin{theorem}
\label{th:main}
When used with the multi-armed strategy INF, the strategy
of Figure~\ref{fig:strat} ensures that
\begin{multline}
\label{eq:bd1}
\sup_{\cF_{L,M}} \oR_T \leq T^{(d+1)/(d+2)} \! \left( 9 \, L^{d/(d+2)} +
5 \left( 2 m\sqrt{\frac{2}{E} \ln \bigl( 2 T^{d+1} \bigr)} \right)^{\!\! d/(d+2)} \right) \\
+ E m^d + 2 \sqrt{2 T d^d} + 1
\end{multline}
as soon as
\begin{equation}
\label{eq:cdtheo}
m \geq \frac{8 M}{L}\,.
\end{equation}
In particular, for
\begin{equation}
\label{eq:alphagamma}
0 < \gamma < \frac{d(d+1)}{(3d+2)(d+2)}
\qquad \mbox{and} \qquad
\alpha = \frac{1}{d+2} \left( \frac{d+1}{d+2} - \gamma \, \frac{3d+2}{d} \right) > 0\,,
\end{equation}
the choices of $m = \lfloor T^\alpha \rfloor$ and $E = m^2 \, \bigl\lceil T^{2 \gamma (d+2)/d} \bigr\rceil$
yield the bound
\begin{equation}
\label{eq:bd2}
\sup_{\cF_{L,M}} \oR_T \leq \max \Biggl\{ \left( \frac{8M}{L} + 1 \right)^{\!\! 1/\alpha}, \ L^{d/(d+2)} \, T^{(d+1)/(d+2)}
\, \bigl( 9 + \varepsilon(T,d) \bigr) \Biggr\}\,,
\end{equation}
where
\[
\varepsilon(T,d) = 5 T^{-\gamma} \bigl( \ln ( 2T^d ) \bigr)^{d/(d+2)} + T^{-\gamma}
+ \frac{ 2\sqrt{2 d^d \, T} + 1}{T^{- (d+1)/(d+2)}}
\]
vanishes as $T$ tends to infinity.
\end{theorem}

Note that the choices of $E$ and $m$ solely depend on $T$, which may however
be unknown in advance; standard arguments, like the doubling trick, can be used to
circumvent the issue, at a minor cost given by an additional constant multiplicative
factor in the bound.

\begin{remark}
There is a trade-off between the value of the constant term
in the maximum, $(1+8M/L)^{1/\alpha}$, and the convergence rate of the vanishing term $\varepsilon(T,d)$ toward
0, which is of order $\gamma$. For instance, in the case $d = 1$, the condition on
$\gamma$ is $0 < \gamma < 2/15$; as an illustration, we get
\begin{itemize}
\item a constant term of a reasonable size, since $1/\alpha \leq 4.87$, when the convergence
rate is small, $\gamma = 0.01$;
\item a much larger constant, since $1/\alpha = 60$, when the convergence rate is faster,
$\gamma = 2/15 - 0.01$.
\end{itemize}
\end{remark}

\begin{proof}
For the strategy INF, as recalled above, $\psi(T',m') = 2 \sqrt{2 T' m'}$.
The bound of Lemma~\ref{lm:3} can thus be instantiated as
\[
\sup_{\cF_{L,M}} \oR_T \leq E m^d +
\E \left[ \frac{L T}{\wt{m}} + 2 \sqrt{ 2 T \wt{m}^d} \right].
\]
We now substitute the definition
\[
\wt{m} = \Bigl\lceil \wt{L}_m^{2/(d+2)} T^{1/(d+2)} \Bigr\rceil
\leq \wt{L}_m^{2/(d+2)} T^{1/(d+2)} \left( 1 + \frac{1}{\wt{L}_m^{2/(d+2)} T^{1/(d+2)}} \right)
\]
and separate the cases depending on whether $\wt{L}_m^{2/(d+2)} T^{1/(d+2)}$ is smaller or larger than $d$
to handle the second term in the expectation.
In the first case, we simply bound $\wt{m}$ by $d$ and get a $2 \sqrt{2 T d^d}$ term.
When the quantity of interest is larger than $d$, then we get the central term in the expectation below by
using the fact that $(1+1/x)^d \leq (1+1/d)^d \leq e$ whenever $x \geq d$.
That is, $\displaystyle{\sup_{\cF_{L,M}} \oR_T}$ is less than
\vspace{-.14cm}

\[
E m^d + \E \left[ T^{(d+1)/(d+2)} \frac{L}{\wt{L}_m^{2/(d+2)}} + 2 \sqrt{ 2 T \, e \Bigl( T^{1/(d+2)} \wt{L}_m^{2/(d+2)} \Bigr)^d }
+ 2 \sqrt{2 T d^d} \, \right].
\]

We will now use the lower and upper bounds on
$\wh{L}_m$ stated by Corollary~\ref{cor:wtL}.
In the sequel we will make repeated use of the following inequality linking
$\alpha$--norms and $1$--norms:
for all integers $p$, all $u_1,\ldots,u_p > 0$, and all $\alpha \in [0,1]$,
\begin{equation}
\label{eq:cvx}
\bigl( u_1 + \ldots + u_p \bigr)^\alpha \leq u_1^\alpha + \ldots + u_p^\alpha\,.
\end{equation}
By resorting to~(\ref{eq:cvx}), we get that with probability at least $1-1/T$,
\begin{eqnarray*}
\lefteqn{ 2 \sqrt{ 2 T \, e \Bigl( T^{1/(d+2)} \wt{L}_m^{2/(d+2)} \Bigr)^d } = 2 \sqrt{ 2 e} \,\, T^{(d+1)/(d+2)} \, \wt{L}_m^{d/(d+2)}} \\
& \leq & 2 \sqrt{ 2 e} \,\, T^{(d+1)/(d+2)} \,
\left( L + 2 m\sqrt{\frac{2}{E} \ln \bigl( 2 m^d T \bigr)} \right)^{d/(d+2)} \\
& \leq & 2 \sqrt{ 2 e} \,\, T^{(d+1)/(d+2)} \,
\left( L^{d/(d+2)} +
\left( 2 m\sqrt{\frac{2}{E} \ln \bigl( 2 m^d T \bigr)} \right)^{d/(d+2)} \right). \\
\end{eqnarray*}
On the other hand, with probability at least $1-1/T$,
\[
\wt{L}_m \geq L - \frac{7 M}{m} \geq \frac{L}{8}\,,
\]
where we assumed that $m$ and $E$ are chosen large enough for the lower bound of Corollary~\ref{cor:wtL}
to be larger than $L/8$. This is indeed the case as soon as
\[
\frac{7M}{m} \leq \frac{7L}{8}\,,
\qquad \mbox{that is,} \qquad
m \geq \frac{8 M}{L}\,,
\]
which is exactly the condition~(\ref{eq:cdtheo}).

Putting all things together (and bounding $m$ by $T$ in the logarithm), with probability
at least $1-1/T$, the regret is less than
\begin{multline}
\label{eq:regprob}
E m^d \ + \ T^{(d+1)/(d+2)} \frac{L}{(L/8)^{2/(d+2)}} \ + \ 2 \sqrt{2 T d^d} \\
+ 2 \sqrt{ 2 e} \,\, T^{(d+1)/(d+2)} \,
\left( L^{d/(d+2)} + \left( 2 m\sqrt{\frac{2}{E} \ln \bigl( 2 T^{d+1} \bigr)} \right)^{d/(d+2)} \right)\,;
\end{multline}
on the event of probability smaller than $\delta = 1/T$ where the above bound does not necessarily hold,
we upper bound the regret by $T$.
Therefore, the expected regret is bounded by~(\ref{eq:regprob}) plus 1.
Bounding the constants as $8^{2/(d+2)} \leq 8^{2/3} = 4$ and $2 \sqrt{ 2 e} \leq 5$
concludes the
proof of the first part of the theorem.
\medskip

The second part follows by substituting the values of $E$ and $m$ in
the expression above and by bounding the regret by $T_0$
for the time steps $t \leq T_0$ for which the condition~(\ref{eq:cdtheo}) is not satisfied.

More precisely, the regret bound obtained in the first part is of the desired order $L^{d/(d+2)}\,T^{(d+1)/(d+2)}$ only if
$E \gg m^2$ and $E m^d \ll T^{(d+1)/(d+2)}$.
This is why we looked for suitable values of $m$ and $E$ in the following form:
\[
m = \lfloor T^\alpha \rfloor
\qquad \mbox{and} \qquad
E = m^2 \, \bigl\lceil T^{2 \gamma (d+2)/d} \bigr\rceil\,,
\]
where $\alpha$ and $\gamma$ are positive.
We choose $\alpha$ as a function of $\gamma$ so that the terms
\begin{align*}
& E m^d = \bigl( \lfloor T^\alpha \rfloor\bigr)^{d+2} \, \bigl\lceil T^{2\gamma (d+2)/d} \bigr\rceil \\
\mbox{and} \qquad &
T^{(d+1)/(d+2)} \left( \frac{m}{\sqrt{E}} \right)^{d/(d+2)}
\, = T^{(d+1)/(d+2)} \Bigl( \bigl\lceil T^{2\gamma (d+2)/d} \bigr\rceil \Bigr)^{-d/(2(d+2))}
\end{align*}
are approximatively balanced; for instance, such that
\[
\alpha(d+2) + 2\gamma (d+2)/d = (d+1)/(d+2) - \gamma\,,
\]
which yields the proposed expression~(\ref{eq:alphagamma}).
The fact that $\alpha$ needs to be positive
entails the constraint on $\gamma$ given in~(\ref{eq:alphagamma}).

When condition~(\ref{eq:cdtheo}) is met, we substitute the values of
$m$ and $E$ into~(\ref{eq:bd1}) to obtain the bound~(\ref{eq:bd2}); the only moment in this substitution
when taking the upper or lower integer parts does not help is for the term
$Em^d$, for which we write (using that $T \geq 1$)
\begin{multline}
\nonumber
E m^d = m^{d+2} \, \bigl\lceil T^{2 \gamma (d+2)/d} \bigr\rceil
\leq T^{\alpha(d+2)} \bigl( 1 + T^{2 \gamma (d+2)/d} \bigr) \\
\leq 2 \, T^{\alpha(d+2)} T^{2 \gamma (d+2)/d} = 2 \, T^{(d+1)/(d+2) - \gamma}\,.
\end{multline}

When condition~(\ref{eq:cdtheo}) is not met, which can only be the case when
$T$ is such that $T^\alpha < 1 + 8M/L$, that is, $T < T_0 = (1+8M/L)^{1/\alpha}$,
we upper bound the regret by $T_0$.
\qed
\end{proof}

\subsection*{Acknowledgements}

This work was supported in part by French National Research Agency (ANR, project
EXPLO-RA, ANR-08-COSI-004) and the PASCAL2 Network of Excellence under EC grant {no.} 216886.

\bibliographystyle{alpha}
\bibliography{ALT-Paper-16-Bubeck-Stoltz-Yu-LipBandit}

\HALversion{
\newpage
\appendix

\section{Proof of Theorem~\ref{th:LB} \\ (Omitted from the Proceedings of ALT'11)}
\label{sec:th:LB}

\begin{proof}
We slightly adapt the (end of the) proof of~\cite[Theorem~13]{HOO};
we take the metric $\ell(\ux,\uy) = L \norm[\ux-\uy]_\infty$.
For $\eps \in (0,1/2)$, the $\eps$--packing number of $[0,1]^d$ with respect to $\ell$ equals
\[
\mathcal{N} \bigl( [0,1]^d, \, \ell, \, \eps \bigr) = \bigl( \lfloor L/\eps \rfloor \bigr)^d \geq 2
\]
provided that $L/\eps \geq 2$, that is, $\eps \leq L/2$. Therefore, Step~5 of the mentioned proof shows that
\[
\sup_{\cF_{L}} \oR_T \geq T \eps \left( 0.5 - 2.2 \, \eps \, \sqrt{ \frac{T}{\bigl( \lfloor L/\eps \rfloor \bigr)^d}} \right)
\]
for all $0 < \eps < \min \{ 1,L \} \big/ 2$. We now optimize this bound.

Whenever $L/\eps \geq \max \{ d,2 \}$, we have
\[
\lfloor L/\eps \rfloor \geq L/\eps - 1 \geq \frac{1}{2} \, L/\eps
\]
in the case where $d = 1$, while for $d \geq 2$,
\[
\bigl( \lfloor L/\eps \rfloor \bigr)^d \geq \bigl( L/\eps - 1 \bigr)^d \geq
\frac{1}{4} \bigl( L/\eps \bigr)^d\,,
\]
where we used the fact that $(1-1/x)^d \geq (1-1/d)^d \geq (1-1/2)^2 = 1/4$ for all $x \geq d$
and $d \geq 2$.

Therefore, whenever $0 < \eps < \min \bigl\{ 1/2, \, L/d, \, L/2 \bigr\}$,
\[
\sup_{\cF_{L}} \oR_T \geq T \eps \left( 0.5 - 4.4 \, \eps^{1+d/2} \, \sqrt{\frac{T}{L^d}}  \right)\,.
\]
We take $\eps$ of the form
\[
\eps = \gamma \, L^{d/(d+2)} \, T^{-1/(d+2)}
\]
for some constant $\gamma < 1$ to be defined later on. The lower bound then equals
\[
\gamma \, L^{d/(d+2)} \, T^{(d+1)/(d+2)} \bigl( 0.5 - 4.4 \, \gamma^{1+d/2} \bigr)
\geq \gamma \, L^{d/(d+2)} \, T^{(d+1)/(d+2)} \bigl( 0.5 - 4.4 \, \gamma^{3/2} \bigr)
\]
where we used the fact that $\gamma < 1$ and $d \geq 1$ for the last inequality.
Taking $\gamma$ such that $0.5 - 4.4 \, \gamma^{3/2} = 1/4$, that is, $\gamma = 1/(4\times4.4)^{2/3}
\geq 0.14$, we get the stated bound.

It only remains to see that the indicated condition on $T$ proceeds from the value
of $\eps$ provided above, the constraint $\eps < \min \bigl\{ 1/2, \, L/d, \, L/2 \bigr\}$,
and the upper bound $\gamma \geq 0.15$.
\qed
\end{proof}
}

\end{document}